\documentclass[11pt,a4paper,reqno]{amsart}
\usepackage{amsfonts}
\usepackage{amssymb}
\usepackage{hyperref}
\usepackage{color}

\numberwithin{equation}{section}
\newtheorem{theorem}{Theorem}[section]
\newtheorem{lemma}[theorem]{Lemma}
\newtheorem{remark}[theorem]{Remark}
\newtheorem{corollary}[theorem]{Corollary}
\newtheorem{proposition}[theorem]{Proposition}

\newtheorem{example}[theorem]{Example}

\allowdisplaybreaks

\def\enddoc{\end{document}}

\def\FRAME#1#2#3#4#5#6#7#8
{
\begin{figure}[ptbh]
	\begin{center}
		\includegraphics[height=#3]{#7}
		\caption{#5}
		\label{#6}
	\end{center}
\end{figure}
}

\begin{document}
	
\title[biharmonic elliptic inequalities]{On positive solutions of biharmonic elliptic inequalities on Riemannian manifolds}

\author[Sun]{Yuhua Sun}
\address{School of Mathematical Sciences and LPMC, Nankai University, 300071 Tianjin, P. R. China}
\email{sunyuhua@nankai.edu.cn}
\thanks{ Sun was supported by the National Natural Science Foundation of China (No.11501303).}

\author[Zheng]{Yadong Zheng}
\address{School of Mathematical Sciences and LPMC, Nankai University, 300071 Tianjin, P. R. China}
\email{yadongzheng2017@sina.com}


	
	
\begin{abstract}
We investigate the non-existence and existence of positive solutions to biharmonic elliptic 
inequalities  on manifolds.
Using Green function and volume growth conditions, we establish the critical exponent for biharmonic problem.
\end{abstract}
	
	\maketitle
	\tableofcontents


\section{Introduction}
Let $M$ be a geodesically complete non-compact Riemannian manifolds with $\text{dim} M>2$, and $\Delta$ be the Laplace-Beltrami operator on $M$.
Further, let $K$ be a compact subset of $M$ such that $M\setminus K$ is a connected domain. We emphasize that $K$ here can be allowed
to be empty, or a singular point of $M$.

Consider the following biharmonic elliptic differential inequality
\begin{equation}\label{BI}
(-\Delta)^2u\geq\Phi(x)u^p\quad\mbox{in $M\setminus K$},
\end{equation}
where 
 $u$ is some unknown nontrivial non-negative $C^4$ function, $p>1$ and $\Phi$ is a given positive function. 

The existence and non-existence of positive solutions of (\ref{BI}) and its related problems have a long history, and have attracted a lot of attentions.
In \cite{Gidas}, Gidas and Spruck considered scalar Lane-Emden equation
\begin{align}\label{GS}
	\Delta u+u^p=0\quad{\rm in}\;\mathbb R^n,
\end{align}
 where $n>2$. They proved that if
$$1<p<\frac{n+2}{n-2},$$
then any non-negative solution of  (\ref{GS}) is zero.
While if $p\geq \frac{n+2}{n-2}$, (\ref{GS}) admits  positive solutions. In particular, for the case $p=\frac{n+2}{n-2}$, all positive solutions of  (\ref{GS}) take the form of
$$u(x)=c_n\left(1+|x|^2\right)^{-\frac{n-2}{2}},$$
with some $c_n>0$.

While if one studied problem (\ref{GS}) in exterior domains $\mathbb R^n\setminus \{0\}$, the critical exponent jumps from $\frac{n+2}{n-2}$ to $\frac{n}{n-2}$. This interesting result  is due to Bidaut-V$\acute{\rm e}$ron  \cite{Bidaut}, more precisely, if
$$1<p\leq \frac{n}{n-2},$$
then  (\ref{GS}) in $\mathbb R^n \setminus \{0\}$ admits no positive solution.  While, for $p>\frac{n}{n-2}$,  the function defined by
$$u(x)=c_{n,p}|x|^{-\frac{2}{p-1}}$$
is a positive solution to (\ref{GS}) for some suitable chosen $c_{n,p}$.

Consider inequality version of (\ref{GS}), namely
\begin{align}\label{GSI}
	\Delta u+u^p\leq0.
\end{align}
 It is worth pointing out that problems (\ref{GSI}) in $\mathbb R^n$ and in exterior domains share the same critical exponent $\frac{n}{n-2}$, see the works of Mitidieri and Pohozaev \cite{Mitidieri}, and
 Bidaut-V$\acute{\rm e}$ron \cite{Bidaut}.  The sharpness of $\frac{n}{n-2}$ can be obtained by  the function
$$u(x)=c_p\left(1+|x|^2\right)^{-\frac{1}{p-1}}$$
for $p>\frac{n}{n-2}$
and small enough $c_p>0$.

Concerning the same problem (\ref{GSI}) but involving non-negative potential term
\begin{align}\label{GSp}
	\Delta u+\Phi(x)u^p\leq0,
\end{align}
where $\Phi(x)\geq C|x|^m$ for large enough $|x|$ and some $m>-2$.
Mitidieri and Pohozaev  in \cite{Mitidieri1} showed the critical exponent of (\ref{GSp}) in $\mathbb R^n$ is $\frac{n+m}{n-2}$.
The same critical exponent of (\ref{GSp}) in $\mathbb R^n\setminus \{0\}$ is obtained by
Bidaut-V$\acute{\rm e}$ron \cite{Bidaut1}.

Now let us turn to the biharmonic problem. The non-existence and existence results for 
\begin{align}\label{B-1}
	(-\Delta)^2 u=u^p\quad{\rm in}\;\mathbb R^n,
\end{align}
with $n>4$ are well established, namely,  if
$$1<p<\frac{n+4}{n-4},$$
then (\ref{B-1}) admits no positive solution.
While for $p=\frac{n+4}{n-4}$, all entire positive
solutions to (\ref{B-1}) can be written in the form
$$u(x)=c_n\left(\frac{c}{1+c^2|x-x_0|^2}\right)^{\frac{n-4}{2}}$$
for some suitable constants $c_n, c>0$, see for example, Lin's work \cite{Lin}, and for general polyharmonic problem, see  Wei and Xu's paper \cite{Wei}. The existence results for $p>\frac{n+4}{n-4}$ were  obtained by Gazzola and Grunau \cite{Gazzola},  Guo and Wei \cite{Guo}.
Further results in this respect can be found in \cite{Cowan, Fazly, Fazly1, Hu, Souplet, Wei1} and the references therein.

Let us move our attention to the inequality version of (\ref{B-1}), that is
\begin{align}\label{B-2}
	(-\Delta)^2 u\geq u^p\quad{\rm in}\;\mathbb R^n,
\end{align}
where $n>4$.
Mitidieri in \cite{Mitidieri0} proved that (\ref{B-2}) admits no positive solution satisfying
 $-\Delta u\geq0$ in $\mathbb R^n$, if
\begin{align}\label{ZB-1}
	1<p\leq\frac{n}{n-4}.
\end{align}
Later, the restriction of super-harmonicity of $u$ can be dropped, see \cite{CDM, MP01}.


The critical exponent for problem (\ref{B-2}) in exterior domains  $\mathbb R^n\setminus\overline{B_1}$ was obtained by P$\acute{\rm e}$rez, Meli$\acute{\rm a}$n and Quaas in \cite{Burgos}, where $B_1$ is the unit ball centered at origin point. Actually, they considered a more general $g(u)$ instead of $u^p$
\begin{align}\label{B-3}
	(-\Delta)^2 u= g(u)\quad{\rm in}\;\mathbb R^n\setminus\overline{B_1},
\end{align}
where $g$ is continuous and nondecreasing in $[0,\infty)$ and $n>4$. They proved that (\ref{B-3}) possesses a positive supersolution $u$ verifying 
	\begin{align}\label{Introduction8}
		-\Delta u>0\quad{\rm in}\;\mathbb R^n\setminus\overline{B_1},
	\end{align}
	if and only if 
\begin{align}\label{Introduction8-1}
	\int_0^\delta\frac{g(s)}{s^{\frac{2(n-2)}{n-4}}}ds<\infty
\end{align}
for any $\delta>0$. As a special case $g(u)=u^p$, condition (\ref{Introduction8-1}) reduces to
 $$p>\frac{n}{n-4}.$$
The approach used in \cite{Burgos} is based on maximum principle and the method of sub and supersolutions which transfers the problem (\ref{B-3})
 into a radially symmetric setting.  
 Recently, the problem $(-\Delta)^2 u=u^p$ in $\mathbb R^n\setminus \overline{ B_1}$ with  Dirichlet boundary and Neumann boundary conditions were
 investigated by  Guo and Liu in \cite{Guo1}, where non-existence results and critical exponent $\frac{n+4}{n-4}$ were established.

Lately in \cite{Aghajani},  Aghajani, Cowan and R$\breve{\rm a}$dulescu investigated the following positive  solutions of the biharmonic problem in a domain $\Omega$ of $\mathbb R^n$ (bounded or not)
\begin{align}\label{Introduction9}
	(-\Delta)^2u\geq\Phi(x)f(u)\quad{\rm in}\;\Omega,
\end{align}
where $\Phi$ and  $f$ are given functions that satisfy certain conditions, and $u$ also needs to satisfy
\begin{align}\label{Introduction10}
	-\Delta u>0\quad{\rm in}\;\Omega.
\end{align}
In particular, when $\Omega=\mathbb R^n\setminus\overline{B_1}$ $(n>4)$, $\Phi(x)=|x|^m$ $(m>-4)$ and $f(u)=u^p$,  then (\ref{Introduction9}) has no positive classical solutions verifying (\ref{Introduction10}) provided
$$0< p\leq\frac{n+m}{n-4}.$$

Let us move our attention from Euclidean space to manifolds. Throughout the paper, let $\mu$ be the Riemannian measure on $M$, $d$ be the geodesic distance, and $B(x,r)$  be the geodesic ball centered at $x$ with radius $r$, that is,
$$B(x,r)=\left\{y\in M: d(x,y)<r\right\},$$
and denote
$$V(x,r)=\mu\left(B(x,r)\right).$$

In the paper \cite{Grigor'yan-sun}, Grigor'yan and Sun considered the following scalar elliptic differential inequality
\begin{align}\label{SE}
	\Delta u+u^p\leq0\quad{\rm in}\;M.
\end{align}
They obtained that if, for some $o\in M$,
\begin{align*}\label{V-SE}
	V(o,r)\leq Cr^{\frac{2p}{p-1}}\left(\ln r\right)^{\frac{1}{p-1}}
\end{align*}
holds for all large enough $r$, then the only non-negative solution of (\ref{SE}) is zero. Moreover, the exponents $\frac{2p}{p-1}$ and $\frac{1}{p-1}$
here
are sharp and can not be relaxed. 

%

Later, the exterior problem of (\ref{SE}) on manifolds is also investigated by Grigor'yan and Sun, see \cite{Grigor'yan}.
 They studied the existence and non-existence of classical solutions to
\begin{equation}\label{SE-1}
\Delta u+\Phi(x)u^p\leq0\quad\mbox{in $M\setminus K$},
\end{equation}
 where $K$ is a compact subset of $M$. Let us fix a reference point $o\in K$.
 Assume the existence of  Green function $G(x,y)$ of $\Delta$ on $M$, and the following hypotheses hold  with given positive reals $\alpha,\gamma$ and $R_0$
\begin{enumerate}
	\item[$(v)$] {There exists $R_0$ such that for all $r>R_0$,
		\begin{align}
		V(o,r)\simeq r^\alpha.\nonumber
	\end{align}}
	\item[$(g)$]{ For all $x,y\in M$ with $d(x,y)>R_0$,
		\begin{align}
		G(x,y)\simeq d(x,y)^{-\gamma}.\nonumber
	\end{align}}
	\item[$(\phi)$]{ There exist reals $m>\gamma-\alpha$ such that, for all $x\in M$ with $d(x,o)> R_0$,
		\begin{align}
			\Phi(x)\gtrsim d(x,o)^m.\nonumber
	\end{align}}
\end{enumerate}	
Grigor'yan and Sun obtained that, if $(v), (g), (\phi)$ are satisfied on $M$, and 
$$1<p\leq\frac{\alpha+m}{\gamma},$$
then any non-negative solution of (\ref{SE-1}) is identical zero. Moreover, if $\dim M>2$ and $M$ has bounded geometry (that is, there exists $\varepsilon>0$ such that the geodesic balls $B(x,\varepsilon)$ on $M$ are uniformly quasi-isometric to the Euclidean ball $B_\varepsilon(0)$ in $\mathbb R^n$), then, for any
$$p>\frac{\alpha}{\gamma},$$
with $\alpha>\gamma$, the inequality (\ref{SE-1}) with $\Phi(x)\equiv1$ admits a positive solution on $M$.

If  the underlying operator $\Delta$ in (\ref{SE-1}) is replaced by biharmonic operator $(-\Delta)^2$, there seems no results concerning biharmonic exterior problem on
manifolds in the literature. Motivated by this problem, and also inspried by the ideas used in \cite{Grigor'yan}, we aim to determine a critical exponent to classify all  positive  solutions of  biharmonic problem in exterior domains of Riemannian manifold, and try to fill this gap in this paper.

In the rest of the paper, let $M$ be a connected non-compact  complete Riemannian manifold  with $\dim M>2$ and have bounded geometry. Let $K$ be a compact subset of $M$ such that $M\setminus K$ is connected, which means that $K$ can be allowed to be empty, or a singular point.
First,  we consider Liouville type theorems for classical positive solutions  of the following problem
\begin{equation}\label{1-1}
	\begin{cases}
	(-\Delta)^2u\geq\Phi(x)u^p&{\rm in}\;M\setminus K,\\
	-\Delta u\geq 0&{\rm in}\;M\setminus K,
	\end{cases}
\end{equation}
where $\Delta$ is the Laplace-Beltrami operator on $M$, $\Phi$ is a given positive function, and $p>1$ is a given exponent. If such a solution $u$ exists, then $u$ is a positive super-harmonic function.

Let $G(x, y)$ be the Green function of $\Delta$ on $M$, that is, the smallest positive fundamental solution of $\Delta$. Assume 
 such $G(x, y)$ exists.
Fix a reference point $o\in K$ (when $K$ is empty, $o$ can be any point on
$M$). Let us introduce the following hypotheses:
\begin{enumerate}
	\item[$(V)$] {There exist positive reals $\alpha$ and $R_0$ such that, for all $x\in M$ and $r\geq R_0$,
		\begin{align}
			V(x,r)\simeq r^\alpha.\nonumber
	\end{align}}
	\item[$(G)$]{ There exist reals $\gamma>\alpha/2$ such that, for all $x,y\in M$ with $d(x,y)\geq R_0$,
		\begin{align}
		G(x,y)\simeq d(x,y)^{-\gamma}.\nonumber
	\end{align}}
	\item[$(\Phi)$]{ There exist reals $m>2(\gamma-\alpha)$ such that, for all $x\in M$ with $d(x,o)\geq R_0$,
		\begin{align}
			\Phi(x)\gtrsim d(x,o)^m.\nonumber
	\end{align}}
\end{enumerate}	

Our first main result is the following non-existence theorem.
\begin{theorem}\label{thm1}
	{\rm Assume that the hypotheses $(V)$, $(G)$ and $(\Phi)$ are satisfied on $M$. If
		$$1<p\leq\frac{\alpha+m}{2\gamma-\alpha},$$
		then  (\ref{1-1}) admits no positive solution.	}
\end{theorem}
\begin{remark}
	\rm{
Our proof relies on the following key ingredients:
\begin{enumerate}\rm{
	\item[(i)] {The existence of Green function $\widetilde{G}$ of $(-\Delta)^2$ on $M$ and then the positive solution's a priori estimate with
		$\widetilde{G}$. More precisely, under the hypotheses $(V)$ and $(G)$, we show that, for any $x,y\in M$ and $x\neq y$ (cf. Proposition \ref{ existence of G} ),
		$$	\widetilde{G}(x,y)=\int_{M}G(x,z)G(z,y)d\mu(z)<\infty.$$
		Then by maximum principle, we have, for any precompact neighborhood $U$ of $K$ with smooth boundary (cf. Lemma \ref{comparison lemma} ),
		\begin{align}
			u(x)\gtrsim \widetilde{G}(x,o)\quad{\rm for\;all}\;x\in \overline{U}^c.\nonumber
	\end{align} }
	\item[(ii)]{The representation formula via Green function $G$ of $\Delta$, that is, by applying the representation formula twice, we obtain from (\ref{1-1}) that
		\begin{align}
			u(x)\geq\int_{\overline{U}^c}{G}_{\overline{U}^c}(x,z)\left(\int_{\overline{U}^c}{G}_{\overline{U}^c}(z,y)\Phi(y)u^p(y)d\mu(y)\right)d\mu(z)\quad{\rm for\;all}\;x\in \overline{U}^c,\nonumber
		\end{align}
		where ${G}_{\overline{U}^c}$  is the Green function of  $\Delta$ in $\overline{U}^c$ with Dirichlet boundary condition.}
	\item[(iii)] {The a priori estimate of positive solution via the first Dirichlet eigenvalue $\lambda_1$ of $\Delta$. By Green formula, we derive, for any precompact open set $\Omega\subset M$ (cf. Lemma \ref{lem1}),
		\begin{equation}
			\inf_{\Omega}\left((-\Delta)^2u-\lambda_1(\Omega)^2u\right)\leq0,\nonumber
		\end{equation}
		where $\lambda_1(\Omega)$ is the first Dirichlet eigenvalue of $\Delta$ in $\Omega$. Hence, it follows that
		$$\inf_{x\in\Omega}\Phi(x)^{\frac{1}{p-1}} u(x)\leq\lambda_1\left(\Omega\right)^{\frac{2}{p-1}}.$$}}
\end{enumerate}	
Combining with the aboves, we can complete the proof of Theorem \ref{thm1} by contradiction argument.}
\end{remark}

For special case $\Phi(x)\equiv1$, we have
\begin{corollary}\label{corollary1}
	{\rm
		Assume that conditions $(V)$ and $(G)$ are satisfied on $M$ with $\gamma<\alpha<2\gamma$. If
		$$1<p\leq\frac{\alpha}{2\gamma-\alpha},$$
		then problem
		\begin{align}\label{ieq1}
			(-\Delta)^2u\geq u^p\quad\text{in }M\setminus K,
		\end{align}
		admits no positive solution $u$  verifying
		\begin{align}\label{C1}
			-\Delta u\geq 0\quad\text{in }M\setminus K.
		\end{align}
	}
\end{corollary}

\begin{example}\label{example1}
	{\rm	Let $M=\mathbb R^n$ $(n>4)$, $K=\emptyset$, or $\{0\}$, or $\overline{B_1}$, $\mu$ be the Lebesgue measure and $d(x,y)=|x-y|$. Then $\Delta$ is the classical Laplacian, and its Green function is given by $$G(x,y)=\frac{c_n}{|x-y|^{n-2}}$$
		with $c_n>0$. It follows that $(V)$ and $(G)$ are satisfied with $\alpha=n$ and $\gamma=n-2$. If, for some $m>-4$,
		$$\Phi(x)\gtrsim (1+|x|)^m,$$
		then by Theorem \ref{thm1}, we know problem
		\begin{equation}
			(-\Delta)^2u\geq\Phi(x)u^p\quad{\rm in}\;\mathbb R^N\setminus K,\nonumber
		\end{equation}
	admits no psotive solution $u$ verifying
		$$-\Delta u\geq0\quad{\rm in}\;\mathbb R^N\setminus K,$$
		provided $$1<p\leq\frac{n+m}{n-4}.$$	
The above is in accordance with the results obtained in \cite{Aghajani, Burgos, Mitidieri0}.
}
\end{example}
\vskip1ex

The second aim of this paper is concerned with the existence of positive
solutions to (\ref{ieq1}). For that, we transfer to consider the following biharmonic equation
\begin{align}\label{thm3-1}
	(-\Delta)^2u=\Psi(x) \left(u^p+l^pF^{ap}\right)\quad\text{in }M,	
\end{align}
where $\Psi, F$ are given positive functions, and $a,l$ are given positive parameters. Clearly, inequality (\ref{ieq1}) is contained in (\ref{thm3-1}) if $\Psi(x)\equiv1$. Hence, in the rest of the paper we concentrate on the existence of positive solutions to  (\ref{thm3-1}) unless otherwise specified.

We introduce  the hypotheses for the functions $\Psi$ and $F$:
\begin{enumerate}
	\item[$(\Psi)$]  {There exist reals $s$ satisfying  $2(\gamma-\alpha)<s\leq0$ such that, for all $x\in M$,
		\begin{align}
			\Psi(x)\simeq\begin{cases}
				R_0^s,&\text{if }d(x,o)\leq R_0,\\ d(x,o)^s,&\text{if }d(x,o)> R_0.
			\end{cases} \nonumber
	\end{align}}
	\item[$(F)$]  {There exist reals $\gamma>\alpha/2$ such that, for all $x\in M$,
		\begin{align}
			F(x)\simeq\begin{cases}
				R_0^{-2\gamma+\alpha},&\text{if }d(x,o)\leq R_0,\\ d(x,o)^{-2\gamma+\alpha},&\text{if }d(x,o)> R_0.
			\end{cases} \nonumber
	\end{align}}
\end{enumerate}	

Our existence result is stated in the following theorem.
\begin{theorem}\label{thm3}{\rm
		Assume that $(V)$, $(G)$, $(\Psi)$ and $(F)$ are satisfied on $M$. If
		$$p>\frac{\alpha+s}{2\gamma-\alpha},$$
		then, for small enough $l$ and $a\in\left(\frac{\alpha+s}{(2\gamma-\alpha)p},1\right)$, 
		 problem (\ref{thm3-1}) admits  a positive solution $u\in C^4(M)$ which satisfies
		\begin{align}\label{thm3-2}
			-\Delta u>0\quad\text{in }M.
		\end{align}
		In particular, $u$  solves
		\begin{align}
			(-\Delta)^2u\geq\Psi(x) u^p\quad\text{in }M.\nonumber
	\end{align}}
\end{theorem}
\vskip1ex

For special case of $\Psi(x)\equiv1$, we derive
\begin{corollary}\label{corollary2}
	{\rm Assume that conditions $(V)$ and $(G)$ are satisfied on $M$ with $\gamma<\alpha<2\gamma$. If $$p>\frac{\alpha}{2\gamma-\alpha},$$
		then (\ref{ieq1}) admits a positive solution $u\in C^4(M)$ which satisfies (\ref{thm3-2}).}
\end{corollary}
\begin{remark}\rm{
In the proof of Theorem \ref{thm3}, we transfer the existence of positive solution $u\in C^4(M)$ to equation (\ref{thm3-1}) with $-\Delta u>0$  to the existence of positive solution $(u,h)\in C^2(M)\times C^2(M)$ to the following system
\begin{align}
	\begin{cases}
		-\Delta u=h&\text{in }M,\\
		-\Delta h=\Psi \left(u^p+l^pF^{ap}\right)&\text{in }M.
	\end{cases}\nonumber
\end{align}
Our main proof can be divided into three steps:
\begin{enumerate}
	\item[(1)] {First, we use the Banach fixed point theorem to obtain a function $u(x)$ which satisfies  the integral equation
		\begin{align}
			u(x)=\int_MG(x,z)\left(\int_MG(z,y)\Psi(y)\left(u(y)^p+l^p F(y)^{ap}\right)d\mu(y)\right)d\mu(z).\nonumber
	\end{align}}
	\item[(2)] {Then, we prove that the function $h(x)$ defined by
		\begin{align}
			h(x):=\int_MG(x,y)\Psi(y)\left(u(y)^p+l^pF(y)^{ap}\right)d\mu(y)\nonumber
		\end{align}
		is H$\ddot{\rm o}$lder continuous and then belongs to $C^2(M)$, and further satisfies
		\begin{align}
			-\Delta h=\Psi(x)\left(u^p+l^pF^{ap}\right)\quad\text{in }M.\nonumber
	\end{align}}
	\item[(3)] {At last, we show the fixed point $u(x)$ is H$\ddot{\rm o}$lder continuous and then belongs to $C^2(M)$, and hence satisfies
		\begin{align}
			-\Delta u=h(x)\quad\text{in }M.\nonumber
	\end{align}}
\end{enumerate}	}
\end{remark}

As a consequence of Corollary \ref{corollary1} and Corollary \ref{corollary2}, we derive a necessary and sufficient criterion for the existence of positive solutions to the problem (\ref{ieq1}) and (\ref{C1}) in exterior domains.
\begin{corollary}\label{corollary3}
	{\rm
		Assume that conditions $(V)$ and $(G)$ are satisfied on $M$ with $\gamma<\alpha<2\gamma$. Then  problem (\ref{ieq1}) in $M\setminus K$
		admits a positive solution $u$ satisfying (\ref{C1}) if and only if $$p>\frac{\alpha}{2\gamma-\alpha}.$$ }
\end{corollary}

\vskip1ex
The rest of the paper is organized as follows: In Section \ref{section2}, we present some useful preliminaries; In Section \ref{section3}, we show the proof of Theorem \ref{thm1}, and Section \ref{section4} is devoted to the proof of Theorem \ref{thm3}.

\textbf{Notations.} In the above and below, the letters $C,C',C_0,C_1,c_0,c_1$... denote positive constants whose values are unimportant and may vary at different occurrences. Moreover, $f\lesssim g$ stands for $f\leq cg$ for a constant $c>0$; $f\gtrsim g$ stands for $f\geq cg$ for a constant $c>0$; $f\simeq g$ means both $f\lesssim g$ and $f\gtrsim g$.

\vskip1ex
\section{Preliminaries}\label{section2}
Let $G(x,y)$ be the Green function of $\Delta$ on $M$, namely the smallest positive fundamental solution of $\Delta$ on $M$.
Throughout the paper, we always assume that $G$ exists. For any precompact open domain $\Omega$ with smooth boundary in $M$, let $G_{\Omega}(x,y)$ be the Green function of $\Delta$ on $\Omega$ satisfying the Dirichlet boundary condition.

Let $\widetilde{G}_\Omega(x,y)$ be the Green function of $(-\Delta)^2$ on $\Omega$, with the singularity at $y\in\Omega$ and  satisfying the boundary data
\begin{equation}
	\widetilde{G}_\Omega|_{\partial\Omega}=0,\quad\Delta\widetilde{G}_\Omega|_{\partial\Omega}=0.\nonumber
\end{equation}
In terms of the Green function $G_{\Omega}(x,y)$, it is well-known that  $\widetilde{G}_\Omega(x,y)$ has the following integral representation (cf. \cite{Sario})
\begin{align}\label{B1}
\widetilde{G}_\Omega(x,y)=\int_{\Omega}G_\Omega(x,z)G_\Omega(z,y)d\mu(z),\end{align}
and satisfies
\begin{equation}
	-\Delta\widetilde{G}_\Omega(x,y)=G_\Omega(x,y).\nonumber
\end{equation}
The global biharmonic Green function $\widetilde{G}(x,y)$ on $M$ is defined by
\begin{equation}
	\widetilde{G}(x,y)=\lim_{\Omega\to M}\widetilde G_\Omega(x,y).\nonumber
\end{equation}
Moreover, if $\widetilde{G}(x,y)<\infty$ for all $x\neq y$, we say  $\widetilde{G}(x,y)$ on $M$ exists.
Here the limits means that we exaust $M$ by a sequence of $\Omega$, actually the existence of the limits is independent of exhaustion
sequence, see \cite{Sario}.

Firstly, we show the existence of biharmonic Green function $\widetilde{G}(x,y)$ on $M$ under the hypotheses $(V)$ and $(G)$ (cf. Proposition \ref{ existence of G}).
Without loss of generality, let us take $R_0=1$.
Introduce two functions
\begin{align}\label{define v}
	v(r)=\begin{cases}
		r^n,& r\leq1,\\ r^\alpha,& r\geq1,
	\end{cases}
\end{align}
and
\begin{align}\label{define g}
	g(r)=\begin{cases}
		r^{2-n},& r\leq1,\\ r^{-\gamma},& r\geq1.
	\end{cases}
\end{align}
Since $M$ has bounded geometry, it follows from $(V)$, $(G)$ and \cite[Lemma 7.1]{ Grigor'yan} that, for all $x,y\in M$ and $r>0$,
\begin{align}\label{properties of V and G}
	V(x,r)\simeq v(r),\quad G(x,y)\simeq g\left(d(x,y)\right).
\end{align}
For our convenience, fix the referenced point $o$, we denote
$$|x|:=d(o,x).$$

The following two lemmas will be used several times in the proof of our results.

\begin{lemma}\label{d}
	\rm{\cite[Lemma 7.2]{ Grigor'yan} If $d(x,y)\geq|y|$, then
	$$d(x,y)\simeq|x|+|y|.$$
}
\end{lemma}

\begin{lemma}\label{F}
	\rm{ If $f$ is a non-negative monotone decreasing function on $(0,\infty)$. Then, for any $x_0\in M$ and $R>0$, we have
		\begin{align}\label{lemma F-1}
			\int_{B(x_0,R)}f\left(d(x_0,x)\right)d\mu(x)\lesssim\int_0^{R}f(r)v(r)\frac{dr}{r},
		\end{align}
	and
	\begin{align}\label{lemma F-2}
		\int_{M\setminus B(x_0,R)}f\left(d(x_0,x)\right)d\mu(x)\lesssim\int_{\frac{1}{2}R}^\infty f(r)v(r)\frac{dr}{r}.
	\end{align}
	In particular,
	\begin{align}\label{lemma F-3}
		\int_Mf\left(d(x_0,x)\right)d\mu(x)\lesssim \int_0^\infty f(r)v(r)\frac{dr}{r}.
	\end{align}}
\end{lemma}
\begin{proof}\rm
By decomposing the integral into a sum of the integrals over the annuli,
and using $V(x_0,r)\lesssim v(r)$ and the monotonicity of $f$, we can finish the proof.
\end{proof}
\begin{proposition}\label{ existence of G}
	{\rm Assume that $(V)$ and $(G)$ are satisfied. Then the biharmonic Green function $\widetilde{G}(x,y)$ exists on $M$ and satisfies
		\begin{align}\label{GGG}
			\widetilde{G}(x,y)=\int_{M}G(x,z)G(z,y)d\mu(z).
		\end{align}
	}
\end{proposition}
\rm \begin{proof}
By (\ref{B1}) and Lebesgue dominated convergence theorem, it suffices to
prove  for all $x,y\in M$ with $x\neq y$, there holds
\begin{align}\label{G<infey}
	\int_{M}G(x,z)G(z,y)d\mu(z)<\infty.
\end{align}
From (\ref{properties of V and G}), we arrive
\begin{align}\label{existence of G1}
\int_{M}G(x,z)G(z,y)d\mu(z)\simeq\int_{M}g\left(d(x,z)\right)g\left(d(y,z)\right)d\mu(z).
\end{align}
We split $M$ into two parts:
$$M_1:=\left\{z\in M:d(x,z)\geq d(y,z)\right\}\quad\text{and}\quad M_2:=\left\{z\in M:d(x,z)< d(y,z)\right\}.$$
For simplicity, we write
$$\rho=d(x,y).$$
Applying Lemma \ref{d}, we obtain
\begin{align}\label{cases d}
	\begin{cases}
		d(x,z)\simeq\rho+d(y,z)&\text{in } M_1,\\
		d(y,z)\simeq\rho+d(x,z)&\text{in } M_2.\\
	\end{cases}
\end{align}
Combining (\ref{existence of G1}) and (\ref{cases d}), and by Lemma \ref{F}, we obtain
\begin{align}
\int_{M}G(x,z)G(z,y)d\mu(z)\simeq&\int_{M_1}g\left(\rho+d(y,z)\right)g\left(d(y,z)\right)d\mu(z)\nonumber\\
&+\int_{M_2}g\left(d(x,z)\right)g\left(\rho+d(x,z)\right)d\mu(z)\nonumber\\\lesssim& \int_0^\infty g(\rho+r)g(r)v(r)\frac{dr}{r}.\nonumber
\end{align}
It remains to verify that
\begin{align}\label{existence of G2}
\int_0^\infty g(\rho+r)g(r)v(r)\frac{dr}{r}<\infty.
\end{align}
Let us consider two cases: $0<\rho<1$ and $\rho\geq1$.

\textbf{Case of $0<\rho<1$}. Using (\ref{define v}) and  (\ref{define g}), we have
\begin{align}\label{existence of G3}
	&\quad\int_0^\infty g(\rho+r)g(r)v(r)\frac{dr}{r}\nonumber\\
	&=\int_0^{1-\rho}(\rho+r)^{2-n}rdr+\int_{1-\rho}^1(\rho+r)^{-\gamma}rdr+\int_1^\infty(\rho+r)^{-\gamma}r^{-\gamma+\alpha-1}dr\nonumber\\
	&\leq\rho^{2-n}\int_0^{1-\rho}rdr+\rho^{-\gamma}\int_{1-\rho}^1rdr+\int_1^\infty r^{-2\gamma+\alpha-1}dr<\infty,
\end{align}
where the last integral in (\ref{existence of G3}) converges due to $\gamma>\alpha/2$. 

\textbf{Case of $\rho\geq1$}. Using (\ref{define v}) and  (\ref{define g}), we have
\begin{align}\label{existence of G4}
	\int_0^\infty g(\rho+r)g(r)v(r)\frac{dr}{r}
	&=\int_0^1(\rho+r)^{-\gamma}rdr+\int_1^\infty(\rho+r)^{-\gamma}r^{-\gamma+\alpha-1}dr\nonumber\\
	&\leq\rho^{-\gamma}\int_0^1rdr+\int_1^\infty r^{-2\gamma+\alpha-1}dr<\infty,
\end{align}
where the last integral in (\ref{existence of G4}) converges due to $\gamma>\alpha/2$.

Hence, we finish the proof of (\ref{existence of G2}). Then
combining with (\ref{G<infey}) and (\ref{existence of G1}), we obtain (\ref{GGG}).
\end{proof}

We also need the following lemmas to prove Theorem \ref{thm1}.

\begin{lemma}\label{comparison lemma}
	\rm{Let $w$ be a non-negative nontrivial function satisfying $(-\Delta)^2w\geq0$ and $-\Delta w\geq0$ in $M\setminus K$. Let $U$ be a precompact neighborhood of $K$ with smooth boundary. For a reference point $o\in K$, there holds
		\begin{align}\label{lem00-1}
			w(x)\gtrsim\widetilde{G}(x,o)\quad{\rm for\;all}\;x\in \overline{U}^c.
		\end{align}}
\end{lemma}
\rm \begin{proof}
	Since $w$ and $-\Delta w$ are super-harmonic outside $K$, by strong maximum principle, we obtain $w$ and $-\Delta w$ are strictly positive.
	Hence, there exists a small enough positive constant $\rho$ such that  $w|_{\partial U}\geq\rho$, and $-\Delta w|_{\partial U}\geq\rho$.
	
	Fix $o\in K$. For any precompact open set $\Omega$ with smooth boundary such that $U\subset\subset\Omega$, denote
	 $$\theta=\frac{1}{\sup\limits_{x\in\partial U}\widetilde{G}_\Omega(x,o)+\sup\limits_{x\in\partial U}{G}_\Omega(x,o)}.$$
	Since
	$$
	\begin{cases}
		(-\Delta)^2 w\geq0&{\rm in}\;\Omega\setminus \overline{U},\nonumber\\
		-\Delta w\geq0&{\rm on}\;\partial\Omega,\nonumber\\
		-\Delta w\geq\rho&{\rm on}\;\partial U
		,\nonumber
	\end{cases}
	$$
	and
	$$
	\begin{cases}
		(-\Delta)^2 \left(\theta\rho\widetilde{G}_\Omega(x,o)\right)=0&{\rm in}\;\Omega\setminus \overline{U},\nonumber\\
		-\Delta \left(\theta\rho\widetilde{G}_\Omega(x,o)\right)=\theta\rho{G}_\Omega(x,o)=0&{\rm on}\;\partial\Omega,\nonumber\\
		-\Delta \left(\theta\rho\widetilde{G}_\Omega(x,o)\right)=\theta\rho{G}_\Omega(x,o)\leq\rho&{\rm on}\;\partial U.
		\nonumber
	\end{cases}
	$$
	By maximum principle, we obtain
	$$-\Delta w\geq-\Delta \left(\theta\rho\widetilde{G}_\Omega(x,o)\right)\quad{\rm in}\;\Omega\setminus \overline{U}.$$
Since
	$$
	\begin{cases}
		w\geq 0= \theta\rho\widetilde{G}_\Omega(x,o)\quad{\rm on}\;\partial\Omega,\nonumber\\
		w\geq\rho\geq\theta\rho\widetilde{G}_\Omega(x,o)\quad{\rm on}\;\partial U,
		\nonumber
	\end{cases}
	$$
  and using  maximum principle again, we obtain
	$$ w\geq \theta\rho\widetilde{G}_\Omega(x,o)\quad{\rm in}\;\Omega\setminus \overline{U}.$$
	Hence by exhausting $M$ with a sequence of $\Omega$, we obtain (\ref{lem00-1}).
\end{proof}

\begin{lemma}\label{lem1}
	\rm{ Let $\Omega$ be a precompact open subset of $M$. Let $\lambda_1(\Omega)$ be the first Dirichlet eigenvalue for $-\Delta$ in $\Omega$. For any non-negative function $f\in C^4(\Omega)\cap C^2(\overline{\Omega})$ satisfies $-\Delta f|_{\partial\Omega}\geq0$, we have
		\begin{equation}\label{2-1}
			\inf_{\Omega}\left((-\Delta)^2f-\lambda_1(\Omega)^2f\right)\leq0.
		\end{equation}				
	}
\end{lemma}

\rm \begin{proof}
We can always assume that $\Omega$ has smooth boundary, otherwise we can use approximation of $\Omega$ from inside by an increasing
	 sequence $\{\Omega_n\}_{n=1}^\infty$ of domains with smooth boundaries, and $\lambda_1(\Omega_n)\to\lambda_1(\Omega)$.
	
	  Let $v$ be the Dirichlet eigenfunction of $-\Delta$ in $\Omega$ with the corresponding eigenvalue $\lambda_1=\lambda_1(\Omega)$. Since $v$ does not change sign in any connected component of $\Omega$,  let us assume that $v>0$ in $\Omega$. Using the fact
	$$
	\begin{cases}
		\Delta v+\lambda_1 v=0&{\rm in}\;\Omega,\nonumber\\
		v=0&{\rm on}\;\partial\Omega
		,\nonumber
	\end{cases}
	$$
and
$$
	\begin{cases}
		(-\Delta)^2 v-\lambda_1^2 v=0&{\rm in}\;\Omega,\nonumber\\
		v=\Delta v=0&{\rm on}\;\partial\Omega
		.\nonumber
	\end{cases}
$$

	Applying Green formula twice, we obtain
	\begin{align}\label{2-2}
		\int_\Omega (-\Delta)^2f vd\mu&=\int_\Omega \Delta v\Delta f d\mu+\int_{\partial\Omega}\left(\frac{\partial\Delta f}{\partial\nu}v-\frac{\partial v}{\partial\nu}\Delta f\right)dS\nonumber\\
		&=\int_\Omega (-\Delta)^2vfd\mu+\int_{\partial\Omega}\left(\frac{\partial f}{\partial\nu}\Delta v-\frac{\partial\Delta v}{\partial\nu} f\right)dS-\int_{\partial\Omega}\frac{\partial v}{\partial\nu}\Delta fdS\nonumber\\
		&=\int_\Omega (-\Delta)^2vfd\mu-\int_{\partial\Omega}\left(\frac{\partial v}{\partial\nu}\Delta f+\frac{\partial\Delta v}{\partial\nu} f\right)dS,
	\end{align}
	where $\nu$ is the outward normal unit vector field on $\partial\Omega$ and $S$ is the surface measure on $\partial\Omega$.
	
	Noting $(-\Delta)^2 v-\lambda_1^2 v=0$, and combining with (\ref{2-2}), we arrive
	\begin{align}
		\int_\Omega (-\Delta)^2fv-\lambda_1^2fvd\mu
		=-\int_{\partial\Omega}\left(\frac{\partial v}{\partial\nu}\Delta f+\frac{\partial\Delta v}{\partial\nu} f\right)dS.\nonumber
	\end{align}
	Since $\frac{\partial v}{\partial\nu}|_{\partial\Omega}\leq0$, $\frac{\partial \Delta v}{\partial\nu}|_{\partial\Omega}\geq0$ and $\Delta f|_{\partial\Omega}\leq0$, it follows that
	\begin{align}
		\int_\Omega \left[(-\Delta)^2f-\lambda_1^2f\right]vd\mu\leq0,\nonumber
	\end{align}
	whence the claim (\ref{2-1}) follows due to $v>0$ in $\Omega$.
\end{proof}

It is not difficult to observe that the condition $(V)$ implies a weaker condition $(V_{\geq})$ as below:
		\begin{enumerate}
			\item[$(V_{\geq})$] {There exist $\tau\in(0,1)$ and $c>0$, such that for all large enough $r$,
				\begin{align}\label{hypotheses-4}
					V(o,r)-V(o,\tau r)\geq cr^\alpha.
			\end{align}}
		\end{enumerate}
	
\begin{lemma}\label{lem2}
	\rm{ Assume $(G)$ is satisfied on $M$.  Let $U$ be a precompact open subset of $M$ with smooth boundary. Let $R$ be large enough and set
\begin{align}
	\begin{cases}
	\Omega=B(o,N^2R), &\Omega_1=B(o,2N^2R),\\
     U_0=B(o,\tau R), &U_1=B(o,r),
	\end{cases}\nonumber
\end{align}
	where $N>2$ is a large enough constant depending on the constants arising from $(G)$,  $\tau\in(0,1)$ from (\ref{hypotheses-4}), $R>\tau^{-1}r$, and $r>2 \text{ diam} U$ such that $$U\subset\subset U_1\subset\subset U_0\subset\subset\Omega\subset\subset\Omega_1.$$
	Then, for all
	$$x\in\Omega\setminus \overline{U_0},\quad y\in\Omega_1\setminus \overline{U_1},$$
	there hold
	$$G(x,y)\gtrsim R^{-\gamma},\quad\text{and}\quad G_{\overline{U}^c}(x,y)\gtrsim R^{-\gamma}.$$
}
\end{lemma}
\rm \begin{proof}
The proof is referred to part of the proof taken from (5.7) to (5.14) of \cite[Theorem 2.1]{ Grigor'yan}.
\end{proof}

\vskip2ex


\section{Non-existence of positive solutions}\label{section3}

\rm \begin{proof}[\rm\textbf{Proof of Theorem \ref{thm1}}]
Assume $u$  is a positive solution of (\ref{1-1}) in $M\setminus K$. Let $U$ be a fixed precompact neighborhood of $K$ with smooth boundary. Fix $o\in K$. By Lemma \ref{comparison lemma}, we obtain
\begin{align}\label{thm1-2}
	u(x)\gtrsim\widetilde{G}(x,o)\quad{\rm for\;all}\;x\in \overline{U}^c.
\end{align}
Since $(-\Delta)^2u\geq\Phi(x)u^p$ and $-\Delta u\geq0$ on $\overline{U}^c$, we have
\begin{align}
	-\Delta u(x)\geq\int_{\overline{U}^c}{G}_{\overline{U}^c}(x,y)\Phi(y)u^p(y)d\mu(y).\nonumber
\end{align}
Moreover,
\begin{align}\label{thm1-3}
	u(x)\geq\int_{\overline{U}^c}{G}_{\overline{U}^c}(x,z)\left(\int_{\overline{U}^c}{G}_{\overline{U}^c}(z,y)\Phi(y)u^p(y)d\mu(y)\right)d\mu(z).
\end{align}
Substituting (\ref{thm1-2}) into (\ref{thm1-3}), we obtain
\begin{align}\label{thm1-4}
	u(x)\gtrsim \int_{\overline{U}^c}{G}_{\overline{U}^c}(x,z)\left(\int_{\overline{U}^c}{G}_{\overline{U}^c}(z,y)\Phi(y)\widetilde{G}^p(y,o)d\mu(y)\right)d\mu(z)\quad{\rm for\;all}\;x\in \overline{U}^c.
\end{align}

Let $U_2$ and $\Omega$ be two precompact open sets  with smooth boundaries  and satisfy
\begin{align}\label{thm1-5}
	U\subset\subset U_2\subset\subset\Omega.
\end{align}
By Lemma \ref{lem1}, we get
$$\inf_{x\in\Omega\setminus \overline{U_2}}\left((-\Delta)^2u-\lambda_1\left(\Omega\setminus\overline{U_2}\right)^2u\right)\leq0.$$
Since $(-\Delta)^2u\geq\Phi u^p$, we obtain
$$\inf_{x\in\Omega\setminus \overline{U_2}}\left(\Phi u^p-\lambda_1\left(\Omega\setminus\overline{U_2}\right)^2u\right)\leq0.$$
It follows that
\begin{align}\label{A1}
	\inf_{x\in\Omega\setminus \overline{U_2}}\Phi(x)^{\frac{1}{p-1}} u(x)\leq\lambda_1\left(\Omega\setminus\overline{U_2}\right)^{\frac{2}{p-1}}.
\end{align}
Substituting (\ref{thm1-4}) into (\ref{A1}), we obtain
\begin{align}\label{A2}
	\inf_{x\in\Omega\setminus \overline{U_2}}\Phi(x)^{\frac{1}{p-1}} \int_{\overline{U}^c}{G}_{\overline{U}^c}(x,z)\left(\int_{\overline{U}^c}{G}_{\overline{U}^c}(z,y)\Phi(y)\widetilde{G}^p(y,o)d\mu(y)\right)d\mu(z)
	\lesssim\lambda_1\left(\Omega\setminus\overline{U_2}\right)^{\frac{2}{p-1}}.
\end{align}
If we can show there exist $U_2$ and $\Omega$ satisfying (\ref{thm1-5}) such that, for any $\varepsilon>0$,
\begin{align}\label{thm1-6}
	\lambda_1\left(\Omega\setminus\overline{U_2}\right)^{\frac{2}{p-1}}<\varepsilon\inf_{x\in\Omega\setminus \overline{U_2}}\Phi(x)^{\frac{1}{p-1}} \int_{\overline{U}^c}{G}_{\overline{U}^c}(x,z)\left(\int_{\overline{U}^c}{G}_{\overline{U}^c}(z,y)\Phi(y)\widetilde{G}^p(y,o)d\mu(y)\right)d\mu(z),
\end{align}
then we obtain a contradiction with (\ref{A2}), and hence, we can conclude that (\ref{1-1}) does not admit any positive solution.

Hence, it remains to prove (\ref{thm1-6}). To this end, we take
\begin{align}
	\begin{cases}
		\Omega=B(o,N^2R), &\Omega_1=B(o,2N^2R),\\
		U_1=B(o,r), &U_2=B(o,\tau R),
	\end{cases}\nonumber
\end{align}
such that $$U\subset\subset U_1\subset\subset U_2\subset\subset\Omega\subset\subset\Omega_1,$$
where the chosen of $N,R,\tau,r$ are the smae as in Lemma \ref{lem2}. Thus, by Lemma \ref{lem2}, for all $x,z\in\Omega\setminus \overline{U_2}$ and $y\in\Omega_1\setminus \overline{U_1}$,
we have
$$
G_{\overline{U}^c}(x,z)\gtrsim R^{-\gamma},\quad G_{\overline{U}^c}(z,y)\gtrsim R^{-\gamma}.$$
It follows that
\begin{align}
\int_{\overline{U}^c}G_{\overline{U}^c}(x,z)G_{\overline{U}^c}(z,y)d\mu(z)\gtrsim R^{-2\gamma}\int_{\Omega\setminus\overline{U_2}}d\mu(z).\nonumber
\end{align}
Hence, for all $x\in\Omega\setminus \overline{U_2}$,
\begin{align}\label{thm1-9}
&\quad\int_{\overline{U}^c}{G}_{\overline{U}^c}(x,z)\left(\int_{\overline{U}^c}{G}_{\overline{U}^c}(z,y)\Phi(y)\widetilde{G}^p(y,o)d\mu(y)\right)d\mu(z)\nonumber\\&
\gtrsim R^{-2\gamma}\left(\int_{\Omega\setminus \overline{U_2}}d\mu(z)\right)\left(\int_{\Omega_1\setminus \overline{U_1}}\Phi(y)\widetilde{G}^p(y,o)d\mu(y)\right).
\end{align}
Next, we claim that, under the hypotheses $(V_{\geq})$, $(G)$ and $(\Phi)$, there hold
\begin{align}\label{thm1-10}
\int_{\Omega\setminus \overline{U_2}}d\mu(z)\gtrsim
R^{\alpha},
\end{align}
and
\begin{align}\label{thm1-10-1}
\int_{\Omega_1\setminus \overline{U_1}}\Phi(y)\widetilde{G}^p(y,o)d\mu(y)\gtrsim
 \begin{cases}
 R^{\alpha+m-p(2\gamma-\alpha)},&{\rm if}\;\alpha+m>p(2\gamma-\alpha),\\
 \ln R,&{\rm if}\;\alpha+m=p(2\gamma-\alpha).
 \end{cases}
\end{align}
Let $R$ be large enough. By $(V_{\geq})$, we obtain
\begin{align}
	\int_{\Omega\setminus \overline{U_2}}d\mu(z)&=\int_{B(o,N^2R)\setminus\overline{ B(o,\tau R)}}d\mu(z)\nonumber\\
	&\geq \int_{B(o,N^2R)\setminus \overline{B(o,\tau N^2R)}}d\mu(z)\nonumber\\
	&= \mu\left(B(o,N^2R)\setminus \overline{B(o,\tau N^2R)}\right)\nonumber\\
	&\gtrsim R^{\alpha},\nonumber
\end{align}
which is exactly (\ref{thm1-10}).

In order to estimate the integral (\ref{thm1-10-1}) in domain $\Omega_1\setminus \overline{U_1}$, let us take $R$ large enough and choose a positive integer $k$ such that
\begin{align}\label{thm1-11}
\tau^{k+1}\geq\frac{r}{N^2R}\geq\tau^{k+2}.
\end{align}
Noting
$$\Omega_1\setminus \overline{U_1}\supset B\left(o,N^2R\right)\setminus\overline{ B\left(o,\tau^{k+1}N^2R\right)},$$
and applying $(\Phi)$, we obtain
\begin{align}\label{GG0}
\int_{\Omega_1\setminus\overline{ U_1}}\Phi(y)\widetilde{G}^p(y,o)d\mu(y)&\geq\sum_{i=0}^{k}\int_{B(o,\tau^iN^2R)\setminus \overline{B(o,\tau^{i+1}N^2R)}}\Phi(y)\widetilde{G}^p(y,o)d\mu(y)
\nonumber\\&\gtrsim\sum_{i=0}^{k}\int_{B(o,\tau^iN^2R)\setminus \overline{B(o,\tau^{i+1}N^2R)}}d(y,o)^m\widetilde{G}^p(y,o)d\mu(y)\nonumber\\&\gtrsim\sum_{i=0}^{k}\int_{B(o,\tau^iN^2R)\setminus \overline{B(o,\tau^{i+1}N^2R)}}\left(\tau^iN^2R\right)^m\widetilde{G}^p(y,o)d\mu(y).
\end{align}
For $y\in B(o,\tau^iN^2R)\setminus \overline{B(o,\tau^{i+1}N^2R)}$, we have by (\ref{define g}) and (\ref{properties of V and G}) that
\begin{align}
\widetilde{G}(y,o)&=\int_MG(y,w)G(w,o)d\mu(w)\nonumber\\
&\gtrsim\int_{M}g\left(d(y,w)\right)g\left(d(w,o)\right)d\mu(w)\nonumber\\
&\gtrsim\int_{B(o,\tau^iN^2R)\setminus \overline{B(o,\tau^{i+1}N^2R)}}g\left(d(y,w)\right)g\left(d(w,o)\right)d\mu(w)\nonumber\\
&\gtrsim\int_{B(o,\tau^iN^2R)\setminus \overline{B(o,\tau^{i+1}N^2R)}}g\left(\tau^iN^2R\right)^2d\mu(w)\nonumber\\& \gtrsim\left(\tau^iN^2R\right)^{-2\gamma}\mu\left(B(o,\tau^iN^2R)\setminus \overline{B(o,\tau^{i+1}N^2R)}\right)\nonumber\\&\gtrsim\left(\tau^iN^2R\right)^{-(2\gamma-\alpha)}.\nonumber
\end{align}
Consequently
\begin{align}\label{GG}
	\widetilde{G}(y,o)^p\gtrsim\left(\tau^iN^2R\right)^{-p(2\gamma-\alpha)}.
\end{align}
Inserting (\ref{GG}) into (\ref{GG0}), we obtain
\begin{align}\label{2}
	\int_{\Omega_1\setminus\overline{ U_1}}\Phi(y)\widetilde{G}^p(y,o)d\mu(y)&\gtrsim\sum_{i=0}^{k}\int_{B(o,\tau^iN^2R)\setminus \overline{B(o,\tau^{i+1}N^2R)}}\left(\tau^iN^2R\right)^{m-p(2\gamma-\alpha)}d\mu(y)\nonumber\\
	&= \sum_{i=0}^{k}\left(\tau^iN^2R\right)^{m-p(2\gamma-\alpha)}\mu\left(B(o,\tau^iN^2R)\setminus \overline{B(o,\tau^{i+1}N^2R)}\right)\nonumber\\
	&\gtrsim\sum_{i=0}^{k}\left(\tau^iN^2R\right)^{\alpha+m-p(2\gamma-\alpha)}.
\end{align}
If $\alpha+m>p(2\gamma-\alpha)$, we have, by $\tau\in(0,1)$ and $k\geq1$,
$$\frac{1-\tau^{(k+1)(\alpha+m-p(2\gamma-\alpha))}}{1-\tau^{\alpha+m-p(2\gamma-\alpha)}}>1.$$
Thus we obtain from (\ref{2})
\begin{align}\label{thm1-15-1}
	\int_{\Omega_1\setminus \overline{U_1}}\Phi(y)\widetilde{G}^p(y,o)d\mu(y)&\gtrsim (N^2R)^{\alpha+m-p(2\gamma-\alpha)}\frac{1-\tau^{(k+1)(\alpha+m-p(2\gamma-\alpha))}}{1-\tau^{\alpha+m-p(2\gamma-\alpha)}}\nonumber\\&\gtrsim R^{\alpha+m-p(2\gamma-\alpha)}.
\end{align}
If $\alpha+m=p(2\gamma-\alpha)$, we obtain from (\ref{2}) that
\begin{align}\label{thm1-15}
	\int_{\Omega_1\setminus \overline{U_1}}\Phi(y)\widetilde{G}^p(y,o)d\mu(y)\gtrsim k.
\end{align}
By (\ref{thm1-11}), we know
$$k\geq\frac{\ln R-\ln\frac{r}{N^2}}{|\ln\tau|}-2.$$
Thus, for large enough $R$, there exists $C=C(N,r,\tau)>0$ such that
\begin{align}\label{thm1-14}
	k\geq C\ln R.
\end{align}
Inserting (\ref{thm1-14}) into (\ref{thm1-15}), we obtain
\begin{align}\label{thm1-16-}
		\int_{\Omega_1\setminus \overline{U_1}}\Phi(y)\widetilde{G}^p(y,o)d\mu(y)\gtrsim\ln R.
\end{align}
Therefore, (\ref{thm1-10-1}) follows by (\ref{thm1-15-1}) and (\ref{thm1-16-}).

A combination of (\ref{thm1-9}), (\ref{thm1-10}), (\ref{thm1-10-1}) and $(\Phi)$ yields that, for large enough $R$,
\begin{align}\label{thm1-17}
&\inf_{x\in\Omega\setminus \overline{U_2}}\Phi(x)^{\frac{1}{p-1}} \int_{\overline{U}^c}{G}_{\overline{U}^c}(x,z)\left(\int_{\overline{U}^c}{G}_{\overline{U}^c}(z,y)\Phi(y)\widetilde{G}^p(y,o)d\mu(y)\right)d\mu(z)
	\nonumber\\
&\gtrsim
	\begin{cases}
		R^{\frac{mp}{p-1}+\alpha-(p+1)(2\gamma-\alpha)},&{\rm if}\;\alpha+m>p(2\gamma-\alpha),\\
			R^{\frac{m}{p-1}+\alpha-2\gamma}\ln R,&{\rm if}\;\alpha+m=p(2\gamma-\alpha).
	\end{cases}
\end{align}
On the other hand, by \cite[Proposition 4.3]{ Grigor'yan}, we have
$$\lambda_1\left(\Omega\setminus\overline{U_2}\right)\lesssim R^{-(\alpha-\gamma)},$$
and thus
\begin{align}\label{thm1-18}
	\lambda_1\left(\Omega\setminus\overline{U_2}\right)^{\frac{2}{p-1}}\lesssim R^{-\frac{2(\alpha-\gamma)}{p-1}}.
\end{align}
Next, we show under the hypothesis
\begin{align}\label{thm1-19}
	1<p\leq\frac{\alpha+m}{2\gamma-\alpha},
\end{align}
the combination of (\ref{thm1-17}) and (\ref{thm1-18}) implies (\ref{thm1-6}).

For the case of $p=\frac{\alpha+m}{2\gamma-\alpha}$, we have
\begin{align}
-\frac{2(\alpha-\gamma)}{p-1}=\frac{m}{p-1}+\alpha-2\gamma.\nonumber
\end{align}
It follows that as $R\to\infty$
\begin{align}
	\lambda_1\left(\Omega\setminus\overline{U_2}\right)^{\frac{2}{p-1}}&=O\left(R^{-\frac{2(\alpha-\gamma)}{p-1}}\right)=o\left(R^{\frac{m}{p-1}+\alpha-2\gamma}\ln R\right)\nonumber\\&=o\left(\inf_{x\in\Omega\setminus \overline{U_2}}\Phi(x)^{\frac{1}{p-1}} \int_{\overline{U}^c}\widetilde{G}_{\overline{U}^c}(x,y)\Phi(y)\widetilde{G}^p(y,o)d\mu(y)\right),\nonumber
\end{align}
Thus (\ref{thm1-6}) follows.

For the case of $1<p<\frac{\alpha+m}{2\gamma-\alpha}$, we claim 
$$-\frac{2(\alpha-\gamma)}{p-1}< \frac{mp}{p-1}+\alpha-(p+1)(2\gamma-\alpha).$$
This is because the above is equivalent to
\begin{align}
	m&>\left[-\alpha-(p+1)(\alpha-2\gamma)-\frac{2(\alpha-\gamma)}{p-1}\right]\frac{p-1}{p}
	\nonumber\\&=\left[(p^2-1)(2\gamma-\alpha)-(p-1)\alpha-2(\alpha-\gamma)\right]\frac{1}{p}
	\nonumber\\&=p(2\gamma-\alpha)-\alpha,\nonumber
\end{align}
Hence, 
we obtain, as $R\to\infty$
\begin{align}
	\lambda_1\left(\Omega\setminus\overline{U_2}\right)^{\frac{2}{p-1}}&=O\left(R^{-\frac{2(\alpha-\gamma)}{p-1}}\right)=o\left(R^{\frac{mp}{p-1}+\alpha+(p+1)(\alpha-2\gamma)}\right)\nonumber\\&=o\left(	\inf_{x\in\Omega\setminus \overline{U_2}}\Phi(x)^{\frac{1}{p-1}} \int_{\overline{U}^c}\widetilde{G}_{\overline{U}^c}(x,y)\Phi(y)\widetilde{G}^p(y,o)d\mu(y)\right),\nonumber
\end{align}
which yields again (\ref{thm1-6}). Thus we complete the proof.
\end{proof}
\vskip2ex
\section{Existence of positive solutions}\label{section4}

In this section we always assume that  $(V)$ and $(G)$ hold on $M$.  Assume also that the hypotheses $(\Psi)$ and $(F)$ are satisfied. Define functions
\begin{align}\label{define psi}
	\psi(r)=\begin{cases}
		1,& r\leq1,\\ r^{s},& r\geq1,
	\end{cases}
\end{align}
 and
\begin{align}\label{define f}
	f(r)=\begin{cases}
		1,& r\leq1,\\ r^{-2\gamma+\alpha},& r\geq1,
	\end{cases}
\end{align}
where $s, \alpha$ satisfy $2(\gamma-\alpha)<s\leq0$ and $\gamma<\alpha<2\gamma$. By $(\Psi)$ and $(F)$, we see that, for all $x \in M$,
\begin{align}\label{properties of psi and F}
	\Psi(x)\simeq \psi(|x|),\quad	F(x)\simeq f(|x|).
\end{align}


\begin{proposition}\label{prop1}{\rm
Assume that the functions $\Psi$ and $F$ satisfying conditions $(\Psi)$ and $(F)$ with $2(\gamma-\alpha)<s\leq0$ and $0<\gamma<\alpha<2\gamma$. If $$p>\frac{\alpha+s}{2\gamma-\alpha},$$  then for all $x\in M$, there hold
\begin{align}\label{prop1-1}
	\int_MG(x,y)\Psi(y)F(y)^{ap}d\mu(y)\lesssim F(x)^b,
\end{align}
and
\begin{align}\label{prop1-2}
	\int_MG(x,y)F(y)^{b}d\mu(y)\lesssim F(x)^a,
\end{align}
where $a,b$ satisfy
\begin{align}\label{ab}
	\begin{cases}
	\frac{\alpha+s}{(2\gamma-\alpha)p}<a<1,\\a+\frac{\alpha-\gamma}{2\gamma-\alpha}<b\leq\frac{\gamma}{2\gamma-\alpha}.
	\end{cases}
\end{align}
}
\end{proposition}
\begin{proof}
From (\ref{ab}), we have
\begin{align}\label{ab1}
	\begin{cases}
		(2\gamma-\alpha)a<\gamma,\\	(2\gamma-\alpha)b\leq\gamma,\\	-s+(2\gamma-\alpha)ap>\alpha,\\	\gamma+(2\gamma-\alpha)(b-a)>\alpha,\\ \gamma-s+(2\gamma-\alpha)(ap-b)>\alpha.
	\end{cases}
\end{align}
Since $\alpha>\gamma$, by $(\ref{ab1})_{4,5}$, we have
\begin{align}\label{ab2}
	\begin{cases}
		(2\gamma-\alpha)(b-a)>0,\\ (2\gamma-\alpha)(ap-b)>s.
	\end{cases}
\end{align}

Fix $x\in M$, and define
$$M_3=\left\{y\in M: |y|\leq d(x,y)\right\},\quad M_4=\left\{y\in M: |y|> d(x,y)\right\}.$$
Applying Lemma \ref{d}, we have
\begin{align}
\begin{cases}
	d(x,y)\simeq|x|+|y| &\text{in } M_3,\\|y|\simeq|x|+d(x,y)&\text{in } M_4.
\end{cases}\nonumber
\end{align}
Together with (\ref{properties of V and G}) and (\ref{properties of psi and F}), the above yields
\begin{align}\label{M1}
	G(x,y)\simeq g\left(d(x,y)\right)\simeq g\left(|x|+|y|\right)\quad\text{in }M_3,
\end{align}
and
\begin{align}\label{M2}
	\begin{cases}
	\Psi(y)\simeq \psi(|y|)\simeq \psi\left(|x|+d(x,y)\right)&\text{in }M_4,
	\\F(y)\simeq f(|y|)\simeq f\left(|x|+d(x,y)\right)&\text{in }M_4.
	\end{cases}
\end{align}

Now we claim
\begin{align}\label{M3-1}
	\int_{M_3}G(x,y)\Psi(y)F(y)^{ap}d\mu(y)\lesssim F(x)^b.
\end{align}
From (\ref{M1}), (\ref{properties of psi and F}) and Lemma \ref{F}, we have
\begin{align}
	\int_{M_3}G(x,y)\Psi(y)F(y)^{ap}d\mu(y)&\simeq\int_{M_3}g\left(|x|+|y|\right)\psi(|y|)f(|y|)^{ap}d\mu(y)\nonumber\\&\lesssim \int_0^\infty g(|x|+r)\psi(r)f(r)^{ap}v(r)\frac{dr}{r}.\nonumber
\end{align}
Hence, it suffices to verify that
\begin{align}\label{step1-3}
	\int_0^\infty g(|x|+r)\psi(r)f(r)^{ap}v(r)\frac{dr}{r}\lesssim f(|x|)^b.
\end{align}
Which is divided into two cases: $|x|\leq1$ and $|x|>1$.

\textbf{Case of $|x|\leq1$}. Using (\ref{define v}), (\ref{define g}), (\ref{define psi}) and (\ref{define f}), we have
\begin{align}\label{step1 case1-1}
	\int_0^\infty g(|x|+r)\psi(r)f(r)^{ap}v(r)\frac{dr}{r}&\leq \int_0^\infty g(r)\psi(r)f(r)^{ap}v(r)\frac{dr}{r}\nonumber\\&=\int_0^1 rdr+\int_1^\infty \frac{dr}{r^{\gamma-s+(2\gamma-\alpha)ap-\alpha+1}}<\infty,
\end{align}
where  we have used that $\gamma-s+(2\gamma-\alpha)ap-\alpha>0$, see $(\ref{ab1})_3$.

\textbf{Case of $|x|\geq1$}. Using (\ref{define v}), (\ref{define g}), (\ref{define psi}) and (\ref{define f}), we have
\begin{align}
	\int_0^\infty g(|x|+r)\psi(r)f(r)^{ap}v(r)\frac{dr}{r}&=\int_0^\infty \frac{1}{(|x|+r)^\gamma}\psi(r)f(r)^{ap}v(r)\frac{dr}{r}\nonumber\\&=\frac{1}{|x|^{(2\gamma-\alpha)b}}\int_0^\infty \frac{|x|^{(2\gamma-\alpha)b}}{(|x|+r)^\gamma}\psi(r)f(r)^{ap}v(r)\frac{dr}{r}
	\nonumber\\&\leq f(|x|)^b\int_0^\infty \psi(r)f(r)^{ap}v(r)\frac{dr}{r}\nonumber\\&= f(|x|)^b\left(\int_0^1 r^{n-1}dr+\int_1^\infty \frac{dr}{r^{-s+(2\gamma-\alpha)ap-\alpha+1}}\right)\nonumber\\&\lesssim f(|x|)^b,\nonumber
\end{align}
where we have used $(2\gamma-\alpha)b\leq\gamma$ and $-s+(2\gamma-\alpha)ap>\alpha$,  see $(\ref{ab1})_{2,3}$.

Hence, we obtain
(\ref{step1-3}), which implies (\ref{M3-1}) holds.

Next, let us prove that
\begin{align}\label{M4-1}
	\int_{M_4}G(x,y)\Psi(y)F(y)^{ap}d\mu(y)\lesssim F(x)^b.
\end{align}
From (\ref{properties of V and G}), (\ref{M2}) and Lemma \ref{F}, we have
\begin{align}
	\int_{M_4}G(x,y)\Psi(y)F(y)^{ap}d\mu(y)&\simeq\int_{M_4}g\left(d(x,y)\right)\psi\left(|x|+d(x,y)\right)f\left(|x|+d(x,y)\right)^{ap}d\mu(y)\nonumber\\&\lesssim \int_0^\infty g(r)\psi(|x|+r)f(|x|+r)^{ap}v(r)\frac{dr}{r}.\nonumber
\end{align}
It suffices to verify that
\begin{align}\label{step2-3}
	\int_0^\infty g(r)\psi(|x|+r)f(|x|+r)^{ap}v(r)\frac{dr}{r}\lesssim f(|x|)^b.
\end{align}
We also divide the proof into two cases: $|x|\leq1$ and $|x|>1$.

\textbf{Case of $|x|\leq1$}. Since
\begin{align}
	\int_0^\infty g(r)\psi(|x|+r)f(|x|+r)^{ap}v(r)\frac{dr}{r}&\leq \int_0^\infty g(r)\psi(r)f(r)^{ap}v(r)\frac{dr}{r}.\nonumber
\end{align}
It follows from (\ref{step1 case1-1}) that (\ref{step2-3}) holds.

\textbf{Case of $|x|\geq1$}. Since
\begin{align}
	&\quad\int_0^\infty g(r)\psi(|x|+r)f(|x|+r)^{ap}v(r)\frac{dr}{r}\nonumber\\
	&=\int_0^\infty g(r) \frac{1}{(|x|+r)^{-s+(2\gamma-\alpha)ap}}v(r)\frac{dr}{r}\nonumber\\&=\int_0^\infty g(r)\frac{1}{(|x|+r)^{(2\gamma-\alpha)b}} \frac{1}{(|x|+r)^{-s+(2\gamma-\alpha)(ap-b)}}v(r)\frac{dr}{r}\nonumber\\&\leq \frac{1}{|x|^{(2\gamma-\alpha)b}}\int_0^\infty g(r) \frac{1}{(1+r)^{-s+(2\gamma-\alpha)(ap-b)}}v(r)\frac{dr}{r}\nonumber\\&\leq f(|x|)^b\left(\int_0^1 rdr+\int_1^\infty\frac{dr}{r^{\gamma-s+(2\gamma-\alpha)(ap-b)-\alpha+1}} \right)\nonumber\\&\lesssim f(|x|)^b,\nonumber
\end{align}
where we have used $(2\gamma-\alpha)(ap-b)>s$ and $\gamma-s+(2\gamma-\alpha)(ap-b)>\alpha$, see $(\ref{ab2})_{2}$ and $(\ref{ab1})_{5}$. 

Hence, we obtain (\ref{step2-3}), and thus (\ref{M4-1}) holds. Combining with (\ref{M3-1}) and (\ref{M4-1}), we obtain
(\ref{prop1-1}).
\vskip1ex
The rest is to finish the proof of (\ref{prop1-2}). Firstly, we show
\begin{align}\label{M3-2}
	\int_{M_3}G(x,y)F(y)^{b}d\mu(y)\lesssim F(x)^a.
\end{align}
As in proof of (\ref{M3-1}), it suffices to prove that
\begin{align}\label{step3-3}
	\int_0^\infty g(|x|+r)f(r)^{b}v(r)\frac{dr}{r}\lesssim f(|x|)^a.
\end{align}
We only need to investigate two cases: $|x|\leq1$ and $|x|>1$.

\textbf{Case of $|x|\leq1$}. Using (\ref{define v}), (\ref{define g}) and (\ref{define f}), we have
\begin{align}\label{step3 case1-1}
	\int_0^\infty g(|x|+r)f(r)^{b}v(r)\frac{dr}{r}&\leq \int_0^\infty g(r)f(r)^{b}v(r)\frac{dr}{r}\nonumber\\&=\int_0^1 rdr+\int_1^\infty \frac{dr}{r^{\gamma+(2\gamma-\alpha)b-\alpha+1}}<\infty,
\end{align}
where the convergence of last integral in (\ref{step3 case1-1}) is due to $\gamma+(2\gamma-\alpha)b>\alpha$, see $(\ref{ab1})_4$. 

\textbf{Case of $|x|\geq1$}. Using (\ref{define v}), (\ref{define g}) and (\ref{define f}), we have
\begin{align}
	\int_0^\infty g(|x|+r)f(r)^{b}v(r)\frac{dr}{r}&=\int_0^\infty \frac{1}{(|x|+r)^\gamma}f(r)^{b}v(r)\frac{dr}{r}\nonumber\\&=\int_0^\infty \frac{1}{(|x|+r)^{(2\gamma-\alpha)a}}\frac{1}{(|x|+r)^{\gamma-(2\gamma-\alpha)a}}f(r)^{b}v(r)\frac{dr}{r}
	\nonumber\\&\leq \frac{1}{|x|^{(2\gamma-\alpha)a}}\int_0^\infty\frac{1}{(1+r)^{\gamma-(2\gamma-\alpha)a}} f(r)^{b}v(r)\frac{dr}{r}\nonumber\\&\leq f(|x|)^a\left(\int_0^1 r^{n-1}dr+\int_1^\infty \frac{dr}{r^{\gamma+(2\gamma-\alpha)(b-a)-\alpha+1}}\right)\nonumber\\&\lesssim f(|x|)^a,\nonumber
\end{align}
where we have used $0<(2\gamma-\alpha)a<\gamma$ and $\gamma+(2\gamma-\alpha)(b-a)>\alpha$, see $(\ref{ab1})_{1,4}$.

Secondly, we show
\begin{align}\label{M4-2}
	\int_{M_4}G(x,y)F(y)^{b}d\mu(y)\lesssim F(x)^a.
\end{align}
It suffices to prove that
\begin{align}\label{step4-3}
	\int_0^\infty g(r)f(|x|+r)^{b}v(r)\frac{dr}{r}\lesssim f(|x|)^a.
\end{align}
We also investigate two cases: $|x|\leq1$ and $|x|>1$.

\textbf{Case of $|x|\leq1$}. Since
\begin{align}
	\int_0^\infty g(r)f(|x|+r)^{b}v(r)\frac{dr}{r}&\leq \int_0^\infty g(r)f(r)^{b}v(r)\frac{dr}{r}.\nonumber
\end{align}
 It follows from (\ref{step3 case1-1}) that (\ref{step4-3}) holds.

\textbf{Case of $|x|\geq1$}. Since
\begin{align}
	\int_0^\infty g(r)f(|x|+r)^{b}v(r)\frac{dr}{r}&=\int_0^\infty g(r) \frac{1}{(|x|+r)^{(2\gamma-\alpha)b}}v(r)\frac{dr}{r}\nonumber\\&=\int_0^\infty g(r)\frac{1}{(|x|+r)^{(2\gamma-\alpha)a}} \frac{1}{(|x|+r)^{(2\gamma-\alpha)(b-a)}}v(r)\frac{dr}{r}\nonumber\\&\leq \frac{1}{|x|^{(2\gamma-\alpha)a}}\int_0^\infty g(r) \frac{1}{(1+r)^{(2\gamma-\alpha)(b-a)}}v(r)\frac{dr}{r}\nonumber\\&\leq f(|x|)^a\left(\int_0^1 rdr+\int_1^\infty\frac{dr}{r^{\gamma+(2\gamma-\alpha)(b-a)-\alpha+1}} \right)\nonumber\\&\lesssim f(|x|)^a,\nonumber
\end{align}
where we have used $(2\gamma-\alpha)(b-a)>0$ and $\gamma+(2\gamma-\alpha)(b-a)>\alpha$, see $(\ref{ab2})_1$ and $(\ref{ab1})_4$. Hence, we obtain (\ref{step4-3}), thus (\ref{M4-2}) holds.

Combining (\ref{M3-2}) and (\ref{M4-2}), we finish the proof of (\ref{prop1-2}).
\end{proof}

\begin{proposition}\label{prop2}{\rm
		Under the assumptions of Proposition \ref{prop1}, for all $x\in M$,  there hold
		\begin{align}\label{prop2-1}
			\int_MG(x,y)\Psi(y)F(y)^{a(p-1)}d\mu(y)\lesssim F(x)^{b-a},
		\end{align}
		and
		\begin{align}\label{prop2-2}
		\sup_{x\in M}\int_MG(x,y)F(y)^{b-a}d\mu(y)<\infty,
		\end{align}
		where $a,b$ satisfy (\ref{ab}).
	}
\end{proposition}
\begin{proof}
	From (\ref{ab}), we have
	\begin{align}\label{ab3}
		\gamma-s+a(p-1)(2\gamma-\alpha)-\alpha>(2\gamma-\alpha)(b-a)>\alpha-\gamma>0.
	\end{align}
	
Fix $x\in M$, and let $M=M_5\cup M_6$, where
$$M_5=\left\{y\in M: |y|\leq 2|x|\right\},\quad M_6=\left\{y\in M: |y|> 2|x|\right\}.$$
By (\ref{define f}), we have
\begin{align}
\frac{1}{f(|y|)}\lesssim \frac{1}{f(|x|)}\quad\text{in } M_5.\nonumber
\end{align}
This together with (\ref{properties of psi and F}) yields
\begin{align}\label{proof prop2-1}
	\frac{1}{F(y)}\lesssim \frac{1}{F(x)}\quad\text{in } M_5.
\end{align}
By Lemma \ref{d}, we have $d(x,y)\simeq|x|+|y|$ in $M_6$ and, hence from (\ref{properties of V and G}),
\begin{align}\label{proof prop2-2}
	G(x,y)\simeq g\left(d(x,y)\right)\simeq g(|x|+|y|)\quad\text{in } M_6.
\end{align}

First, we show
\begin{align}\label{prop2 step1-1}
	\int_{M_5}G(x,y)\Psi(y)F(y)^{a(p-1)}d\mu(y)\lesssim F(x)^{b-a}.
\end{align}
Using (\ref{proof prop2-1}) and Proposition \ref{prop1}, we have
\begin{align}
	\int_{M_5}G(x,y)\Psi(y)F(y)^{a(p-1)}d\mu(y)&\lesssim \int_{M_5}G(x,y)\Psi(y)\frac{F(y)^{ap}}{F(x)^{a}}d\mu(y)\nonumber\\&\lesssim F(x)^{b-a}.\nonumber
\end{align}
Hence, we obtain (\ref{prop2 step1-1}).

Second, we show
\begin{align}\label{M6-1}
	\int_{M_6}G(x,y)\Psi(y)F(y)^{a(p-1)}d\mu(y)\lesssim F(x)^{b-a}.
\end{align}
Using (\ref{properties of psi and F}), (\ref{proof prop2-2}) and Lemma \ref{F}, we have
\begin{align}
	\int_{M_6}G(x,y)\Psi(y)F(y)^{a(p-1)}d\mu(y)&\simeq\int_{M_6}g(|x|+|y|)\psi(|y|)f(|y|)^{a(p-1)}d\mu(y)\nonumber\\&\lesssim \int_{|x|}^\infty g(|x|+r)\psi(r)f(r)^{a(p-1)}v(r)\frac{dr}{r}\nonumber\\
	&\lesssim \int_{|x|}^\infty g(r)\psi(r)f(r)^{a(p-1)}v(r)\frac{dr}{r}.\nonumber
\end{align}
It remains to verify that
\begin{align}\label{prop2 step2-3}
	\int_{|x|}^\infty g(r)\psi(r)f(r)^{a(p-1)}v(r)\frac{dr}{r}\lesssim f(|x|)^{b-a}.
\end{align}
Let us divide into two cases: $|x|\leq1$ and $|x|>1$.

\textbf{Case of $|x|\leq1$}. Using (\ref{define v}), (\ref{define g}), (\ref{define psi}) and (\ref{define f}), we have
\begin{align}
	\int_{|x|}^\infty g(r)\psi(r)f(r)^{a(p-1)}v(r)\frac{dr}{r}&\leq\int_0^\infty g(r)\psi(r)f(r)^{a(p-1)}v(r)\frac{dr}{r}\nonumber\\&=\int_0^1rdr+\int_1^\infty\frac{dr}{r^{\gamma-s+a(p-1)(2\gamma-\alpha)-\alpha+1}}<\infty,\nonumber
\end{align}
where we have used that $\gamma-s+a(p-1)(2\gamma-\alpha)-\alpha>0$, see (\ref{ab3}). 

\textbf{Case of $|x|\geq1$}. Using (\ref{define v}), (\ref{define g}), (\ref{define psi}) and (\ref{define f}), we have
\begin{align}
	\int_{|x|}^\infty g(r)\psi(r)f(r)^{a(p-1)}v(r)\frac{dr}{r}&=\int_{|x|}^\infty \frac{dr}{r^{\gamma-s+a(p-1)(2\gamma-\alpha)-\alpha+1}}\nonumber\\&\lesssim \frac{1}{|x|^{\gamma-s+a(p-1)(2\gamma-\alpha)-\alpha}}\nonumber\\&\lesssim \frac{1}{|x|^{(2\gamma-\alpha)(b-a)}}=f(|x|)^{b-a},\nonumber
\end{align}
where we have used $\gamma-s+a(p-1)(2\gamma-\alpha)-\alpha>(2\gamma-\alpha)(b-a)>0$, see (\ref{ab3}). Hence, we obtain (\ref{prop2 step2-3}).
Combining (\ref{prop2 step1-1}) and (\ref{M6-1}), we finish the proof of (\ref{prop2-1}).

Now we give the proof of (\ref{prop2-2}). First, we show
\begin{align}\label{prop2 step3-1}
	\int_{M_5}G(x,y)F(y)^{b-a}d\mu(y)<\infty.
\end{align}
Using (\ref{proof prop2-1}) and Proposition \ref{prop1}, we have
\begin{align}
	\int_{M_5}G(x,y)F(y)^{b-a}d\mu(y)&\lesssim \int_{M_5}G(x,y)\frac{F(y)^{b}}{F(x)^{a}}d\mu(y)<\infty.\nonumber
\end{align}
Hence  (\ref{prop2 step3-1}) follows.

Next, we show
\begin{align}\label{prop2 step4-1}
	\int_{M_6}G(x,y)F(y)^{b-a}d\mu(y)<\infty.
\end{align}
Using (\ref{properties of psi and F}), (\ref{proof prop2-2}) and Lemma \ref{F}, we obtain
\begin{align}
	\int_{M_6}G(x,y)F(y)^{b-a}d\mu(y)&\simeq\int_{M_6}g(|x|+|y|)f(|y|)^{b-a}d\mu(y)\nonumber\\&\lesssim \int_0^\infty g(r)f(r)^{b-a}v(r)\frac{dr}{r}\nonumber\\&= \int_0^1rdr+\int_1^\infty\frac{dr}{r^{\gamma+(2\gamma-\alpha)(b-a)-\alpha+1}}<\infty,\nonumber
\end{align}
where  $\gamma+(2\gamma-\alpha)(b-a)>\alpha$ is used, see (\ref{ab3}). Hence, we obtain (\ref{prop2 step4-1}).

Then combining (\ref{prop2 step3-1}) and (\ref{prop2 step4-1}), we complete the proof of (\ref{prop2-2}).
\end{proof}

\rm \begin{proof}[\rm\textbf{Proof of Theorem \ref{thm3}}]
We divide the proof into four steps.

\textbf{Step 1} We show the operator $T$ defined by
\begin{align}\label{define Tu}
	Tu(x):=\int_MG(x,z)\left(\int_MG(z,y)\Psi(y)\left(u(y)^p+l^p F(y)^{ap}\right)d\mu(y)\right)d\mu(z)
\end{align}
is a contraction map acting on the space
$$S_l=\left\{u\in L^\infty(M):0\leq u(x)\leq l F(x)^a\right\},$$
where $l$ is a positive small enough constant to be chosen later, and $a$ is defined by $(\ref{ab})_1$.

Notice that $S_l$ is a closed set of $L^\infty(M)$. Let us show that $$TS_l\subset S_l.$$
By Proposition \ref{prop1}, we have
\begin{align}
	Tu(x)&=\int_MG(x,z)\left(\int_MG(z,y)\Psi(y)\left(u(y)^p+l^p F(y)^{ap}\right)d\mu(y)\right)d\mu(z)\nonumber\\&\leq2l^p\int_MG(x,z)\left(\int_MG(z,y)\Psi(y)F(y)^{ap}d\mu(y)\right)d\mu(z)\nonumber\\&\leq2Cl^p\int_MG(x,z)F(z)^bd\mu(z)\nonumber\\&\leq2Cl^pF(x)^a.\nonumber
\end{align}
By choosing $l$ small enough such that $2Cl^p\leq l$, we obtain $Tu\in S_l$ and hence $TS_l\subset S_l.$

Then, let us show that $T$ is a contraction map. For $u_1,u_2\in S_l$, we have
\begin{align}
	\left|Tu_1-Tu_2\right|\leq\int_MG(x,z)\left(\int_MG(z,y)\Psi(y)\left|u_1(y)^p-u_2(y)^p\right|d\mu(y)\right)d\mu(z).\nonumber
\end{align}
Noting that
$$|u_1^p-u_2^p|\leq p\sup\{u_1^{p-1},u_2^{p-1}\}|u_1-u_2|,$$
we obtain
\begin{align}
	\left|Tu_1-Tu_2\right|\leq pl^{p-1}\|u_1-u_2\|_{L^\infty}\int_MG(x,z)\left(\int_MG(z,y)\Psi(y)F(y)^{a(p-1)}d\mu(y)\right)d\mu(z).\nonumber
\end{align}
Applying Proposition \ref{prop2}, we obtain
\begin{align}
	\|Tu_1-Tu_2\|_{L^\infty}&\leq C pl^{p-1}\|u_1-u_2\|_{L^\infty}\int_MG(x,z)F(z)^{b-a}d\mu(z)\nonumber\\&\leq C'pl^{p-1}\|u_1-u_2\|_{L^\infty}.\nonumber
\end{align}
Choosing $l$ small enough such that $C'pl^{p-1}<1$, hence $T$ is a contraction map. By the Banach Fixed Point Theorem, $T$ has a fixed point $u$.
In the rest of the proof, we verify that the fixed point $u$ belongs to $C^4(M)$ and satisfies (\ref{thm3-1}).

\textbf{Step 2} Let us show that the function $h(x)$ defined by
\begin{align}\label{define U}
	h(x):=\int_MG(x,y)\Psi(y)\left(u(y)^p+l^pF(y)^{ap}\right)d\mu(y)
\end{align}
is $C^2$ in $M$ and satisfies
\begin{align}\label{equation U}
	-\Delta h=\Psi(x)\left(u^p+l^pF^{ap}\right)\quad\text{in }M.
\end{align}

Denote
\begin{align}\label{define w}
	w:=\Psi(x)\left(u^p+l^pF^{ap}\right).
\end{align}
Then by (\ref{define U})
\begin{align}\label{rewrite U}
	h(x)=\int_MG(x,y)w(y)d\mu(y).
\end{align}
Since $u\in S_l$, we have
\begin{align}\label{w}
	w\leq2l^p\Psi F^{ap},
\end{align}
which implies
\begin{align}\label{w satisfies}
	w(x)\leq C(1+|x|)^{s-(2\gamma-\alpha)ap}.
\end{align}
Here by $(\ref{ab})_1$ we know  $s-(2\gamma-\alpha)ap<0$.

Let us first prove that $h$ is locally H$\ddot{\rm o}$lder, that is, there exist $\theta\in(0,1)$ depending on $n,\alpha,\gamma$ and the bounded geometry constants (see (\ref{theta})) 
$$\left|h(x)-h(x')\right|\lesssim d(x,x')^\theta,$$
provided $d(x,x')$ is small enough. Set
$$\varepsilon:=d(x,x')^{1/N},$$
with $N>2$. Assume that $d(x,x')$ is so small that
$$\varepsilon<\frac{1}{4}\min\left\{1,|x|^{-1}\right\}.$$
It follows that $d(x,x')=\varepsilon^N<\frac{1}{4}\varepsilon$, hence $x'\in B\left(x,\frac{1}{4}\varepsilon\right)$. Set
\begin{align}\label{define R}
	R:=\varepsilon^{-1}>4\max\{1,|x|\},
\end{align}
and observe that
\begin{align}
	\left|h(x)-h(x')\right|&\leq\int_{B(x,2\varepsilon)}\left|G(x,y)-G(x',y)\right|w(y)d\mu(y)\label{U-U'1}\\&+\int_{M\setminus B(x,R)}\left|G(x,y)-G(x',y)\right|w(y)d\mu(y)\label{U-U'2}\\&+\int_{B(x,R)\setminus B(x,2\varepsilon)}\left|G(x,y)-G(x',y)\right|w(y)d\mu(y)\label{U-U'3}.
\end{align}
For the integral in (\ref{U-U'1}), since $y\in B(x,2\varepsilon)$ and $x'\in B\left(x,\frac{1}{4}\varepsilon\right)$, we have
$$\{y: y\in B(x,2\varepsilon)\}\subset\{y:y\in B(x',4\varepsilon)\}.$$
Using the boundedness of $w$, by (\ref{properties of V and G}) and Lemma \ref{F}, we have
\begin{align}
	\int_{B(x,2\varepsilon)}\left|G(x,y)-G(x',y)\right|w(y)d\mu(y)&\lesssim \int_{B(x,2\varepsilon)}\left(G(x,y)+G(x',y)\right)d\mu(y)\nonumber\\&\simeq\int_{B(x,2\varepsilon)}\left(g\left(d(x,y)\right)+g\left(d(x',y)\right)\right)d\mu(y)\nonumber\\&\lesssim \int_{B(x,2\varepsilon)}g\left(d(x,y)\right)d\mu(y)+\int_{B(x',4\varepsilon)}g\left(d(x',y)\right)d\mu(y) \nonumber\\&\lesssim \int_0^{\varepsilon}g(r)v(r)\frac{dr}{r}+\int_0^{2\varepsilon}g(r)v(r)\frac{dr}{r}\lesssim \varepsilon^2.\nonumber
\end{align}
In order to estimate the integral in (\ref{U-U'2}), observe that, for $y\in M\setminus B(x,R)$, we have by (\ref{properties of V and G})
\begin{align}
	\left|G(x,y)-G(x',y)\right|\lesssim g\left(d(x,y)\right)+g\left(d(x',y)\right)\lesssim R^{-\gamma},\nonumber
\end{align}
and by (\ref{w satisfies}) and (\ref{define R}), we derive
$$w(y)\lesssim |y|^{s-(2\gamma-\alpha)ap}\lesssim \left(d(x,y)-|x|\right)^{s-(2\gamma-\alpha)ap}\lesssim \left(\frac{3}{4}d(x,y)\right)^{s-(2\gamma-\alpha)ap}.$$
Noting $a>\frac{\alpha+s}{(2\gamma-\alpha)p}$, we obtain by Lemma \ref{F}
\begin{align}
\int_{M\setminus B(x,R)}\left|G(x,y)-G(x',y)\right|w(y)d\mu(y)&\lesssim R^{-\gamma}\int_{M\setminus B(x,R)}d(x,y)^{s-(2\gamma-\alpha)ap}d\mu(y)\nonumber\\&\lesssim R^{-\gamma}\int_{\frac{1}{2}R}^\infty r^{s-(2\gamma-\alpha)ap+\alpha-1}dr\nonumber\\&\lesssim R^{-\gamma+s-(2\gamma-\alpha)ap+\alpha}\nonumber\\&\lesssim  R^{-\gamma}=\varepsilon^\gamma.\nonumber
\end{align}

If $y\in B(x,R)\setminus B(x,2\varepsilon)$, then the function $G(\cdot,y)$ is harmonic in $B(x,\varepsilon)$. Using \cite[Theorem 8.22]{Gilbarg}, we obtain
\begin{align}
	\sup_{z\in B(x,\varepsilon^N)}G(z,y)-\inf_{z\in B(x,\varepsilon^N)}G(z,y)\lesssim \varepsilon^{(N-1)\eta}\sup_{z\in B(x,\varepsilon)}G(z,y),\nonumber
\end{align}
where  $\eta<1$ are positive constants depend on the bounded geometry constants and on $n$. Using also $\varepsilon^N=d(x,x')$, we obtain
\begin{align}\label{G-G'1}
	\left|G(x,y)-G(x',y)\right|&\leq\sup_{z\in B(x,\varepsilon^N)}G(z,y)-\inf_{z\in B(x,\varepsilon^N)}G(z,y)\nonumber\\&\lesssim \varepsilon^{(N-1)\eta}\sup_{z\in B(x,\varepsilon)}G(z,y).
\end{align}
By (\ref{properties of V and G}), we obtain
\begin{align}\label{sup G}
	\sup_{z\in B(x,\varepsilon)}G(z,y)\lesssim \varepsilon^{2-n}.
\end{align}
Combining (\ref{G-G'1}) and (\ref{sup G}), using also the boundedness of $w$, we obtain, for the integral in (\ref{U-U'3}),
\begin{align}
	\int_{B(x,R)\setminus B(x,2\varepsilon)}\left|G(x,y)-G(x',y)\right|w(y)d\mu(y)&\lesssim R^\alpha\varepsilon^{(N-1)\eta+2-n}\nonumber\\&=\varepsilon^{(N-1)\eta+2-n-\alpha}.\nonumber
\end{align}

Combing all the above estimates, we obtain from (\ref{U-U'1}), (\ref{U-U'2}) and (\ref{U-U'3}) that
$$
	\left|h(x)-h(x')\right|\lesssim \varepsilon^2+\varepsilon^\gamma+\varepsilon^{(N-1)\eta+2-n-\alpha}\lesssim d(x,x')^\theta,
$$
where
\begin{align}\label{theta}
\theta=\frac{1}{N}\min\{2,\gamma,(N-1)\eta+2-n-\alpha\}.
\end{align}
Choosing $N>2$, and noting $0<\eta<1$, we derive $0<\theta<1$.

Since $F$ is locally H$\ddot{\rm o}$lder continuous, we obtain from (\ref{define w}) that $w$ is locally H$\ddot{\rm o}$lder on $M$. For any precompact domain $\Omega\subset M$, we obtain by \cite[Lemma 8.1]{ Grigor'yan} that the function
\begin{align}
	h_\Omega(x)=\int_\Omega G_\Omega(x,y)w(y)d\mu(y)\nonumber
\end{align}
belongs to $C^2(\Omega)$. Since the difference $h-h_\Omega$ is harmonic in $\Omega$ in the distributional sense, it follows that $h-h_\Omega$ has a smooth modification in $\Omega$. Therefore, $h$ has a $C^2$-modification in $\Omega$. Since $h$ is continuous, we conclude that $h\in C^2(\Omega)$. Since $\Omega$ is arbitrary, it follows that $h\in C^2(M)$. By \cite[Lemma 13.1]{A. Grigor'yan}, we derive $h$ solves $-\Delta h=w$, which is equivalent to (\ref{equation U}).

\textbf{Step 3} Let us show that the fixed point $u$ of $T$ belongs to $ C^2(M)$ and satisfies
\begin{align}\label{equation u and U}
	-\Delta u=h(x)\quad\text{in }M.
\end{align}
From (\ref{define Tu}) and (\ref{define U}), we derive
$$u(x)=\int_MG(x,y)h(y)d\mu(y).$$
Observing (\ref{rewrite U}) and (\ref{w}), we have by Proposition \ref{prop1} that
$$h\lesssim F^b,$$
which implies
\begin{align}\label{U satisfies}
	h(x)\lesssim (1+|x|)^{-(2\gamma-\alpha)b}.
\end{align}

Applying the same arguments as in Step 2, we can obtain that $u$ is locally H$\ddot{\rm o}$lder.
From Step 2, we derive that $h$ is locally H$\ddot{\rm o}$lder on $M$. For any precompact domain $\Omega\subset M$, we obtain by \cite[Lemma 8.1]{ Grigor'yan} again that the function
\begin{align}
	u_\Omega(x)=\int_\Omega G_\Omega(x,y)h(y)d\mu(y)\nonumber
\end{align}
belongs to $C^2(\Omega)$. Since the difference $u-u_\Omega$ is harmonic in $\Omega$ in the distributional sense, it follows that $u-u_\Omega$ has a smooth modification in $\Omega$. Therefore, $u$ has a $C^2$-modification in $\Omega$. Since $u$ is continuous, we conclude that $u\in C^2(\Omega)$. Since $\Omega$ is arbitrary, it follows that $u\in C^2(M)$. By \cite[Lemma 13.1]{A. Grigor'yan}, we obtain that  $-\Delta u=h$.

\textbf{Step 4} Let us prove that the fixed point $u$ of $T$ belongs to $ C^4(M)$ and satisfies (\ref{thm3-1}).

From Steps 2 and 3, we see that $(u,h)\in C^2(M)\times C^2(M)$ and satisfies
\begin{align}
	\begin{cases}
		-\Delta u=h(x)&\text{in }M,\\
-\Delta h=\Psi(x) \left(u^p+l^pF^{ap}\right)&\text{in }M.
	\end{cases}\nonumber
\end{align}
This implies immediately that $u\in C^4(M)$ and
$$(-\Delta)^2u=\Psi(x) \left(u^p+l^pF^{ap}\right)\quad\text{in }M,$$
which is exactly (\ref{thm3-1}).
\end{proof}

\bibliographystyle{model1a-num-names}

\end{document}